\documentclass{amsproc}
\usepackage{amssymb}
\usepackage{graphicx}
\usepackage{amscd}
\usepackage{amsmath}
\usepackage{amsfonts}
\newtheorem{acknowledgement}{Acknowledgement}

\newtheorem{case}{Case}

\newtheorem{lemma}{Lemma}

\newtheorem{proposition}{Proposition}
\newtheorem{remark}{Remark}

\newtheorem{Theorem}{Theorem}
\numberwithin{equation}{section}

\begin{document}
\title[Yamabe invariant]{Free boundary hypersurfaces with nonpositive Yamabe invariant in mean convex manifolds}
\author{A. Barros}
\address{Departamento de Matem\'{a}tica-UFC\\
60455-760-Fortaleza-CE-Br}
\email{abbarros@mat.ufc.br}

\author{C.Cruz}
\address{Departamento de Matem\'{a}tica-UFC\\
60455-760-Fortaleza-CE-Br}
\email{ tiarlos@alu.ufc.br}

\thanks{Authors partially supported by CNPq-Brazil}

 \urladdr{http://www.mat.ufc.br}

\maketitle

\begin{abstract}
We obtain some estimates on the area of the boundary and on the volume of a certain
free boundary hypersurface $\Sigma$ with nonpositive Yamabe invariant in a Riemannian
$n$-manifold with bounds for the scalar curvature and the mean curvature of the boundary. 
Assuming further that $\Sigma$ is locally volume-minimizing in a manifold $M^n$ with 
scalar curvature bounded below by a nonpositive constant and mean convex boundary, 
we conclude that locally $M$ splits along $\Sigma$. In the case that the scalar curvature of $M$ is at least $-n(n-1)$ and
$\Sigma$ locally minimizes a certain 
functional inspired by \cite{Y}, a neighborhood of $\Sigma$ in $M$ is isometric to $((-\varepsilon,\varepsilon)\times\Sigma,dt^2+e^{2t}g)$,
where $g$ is Ricci flat.

\end{abstract}

\section{Introduction and main results}

In recent years, rigidity involving the scalar curvature has been studied because these problems are motivated by general relativity and have strong
connections with the theory of minimal surfaces. Moreover, the existence of an area-minimizing surface of some kind, enables us to deduce several
rigidity theorems. A deeper result due to Schoen and Yau \cite{SY} asserts that any area-minimizing
surface in a three-manifold $(M, g)$ with positive scalar curvature is homeomorphic either to $\mathbb{S}^2 $ or $\mathbb{RP}^2 $.
Motivated by this, the rigidity of area-minimizing projective planes was studied by Bray et al. \cite{BBEN},
while the case of area-minimizing two-spheres was obtained by Bray, Brendle and Neves in \cite{BBN}. It was also observed by Cai and
Galloway \cite{CG} that a three-manifold with nonnegative scalar curvature is flat in a neighborhood of a two-sided embedded two-torus
which is locally area-minimizing.
 For surfaces of genus $g(\Sigma)>1 $, Nunes \cite[Theorem 3]{N} has obtained an interesting rigidity result for minimal hyperbolic
 surfaces in three-manifolds with scalar curvature bounded by a negative constant. There is also an
  unified point of view with alternative proofs about these cases considered by Micallef and Moraru \cite{MM}.
 For a good reference about other rigidity theorems we refer the reader to \cite{B}.
 
 In higher dimensions, Cai \cite{C} showed a local splitting of an $n$-dimensional manifold $M$ with nonegative
escalar curvature containing a volume-minimizing hypersurface that does not admit a metric of positive
scalar curvature. In this direction, Moraru \cite{M} proved a natural extension of the rigidity result contained in \cite{N}.

In this paper we are interested in studying rigidity of hypersurfaces with boundary. We point out that the boundary geometry can influence the
geometry of the manifold. For example, there is a relationship between the topology of free boundary minimal surfaces and the geometry of the
ambient manifold, such as convexity of the boundary and bounds on the scalar curvature, by means of the second variation of area.
Very recently, Ambrozio \cite{A} established theorems of rigidity for area-minimizing free boundary surfaces in mean convex three-manifolds.
Moreover, if the ambient $M$ has a lower bound on its scalar curvature by a negative constant, there is a rigidity theorem for solutions 
of the Plateau problem for certain homotopically
non-trivial curves in $\partial M$ with length-minimizing boundary.

In order to state our main results we need to introduce the Yamabe invariant for manifolds with boundary.
Let $(\Sigma,g)$ be a compact Riemannian manifold $n\geq3$ with nonempty boundary $\partial \Sigma$.
For $(a,b)\in\mathbb{R}\times\mathbb{R}-\{(0,0)\}$, we define the following functional

\begin{equation}\label{yama}
Q^{a,b}_g(\varphi)=\frac{\int_{\Sigma}\big(\frac{4(n-1)}{n-2}\|\nabla\varphi\|_g^2+R_g\varphi^2\big)d\sigma +2\int_{\partial \Sigma}\kappa_g\varphi^2d\sigma_
{\partial \Sigma}}{\big(a(\int_{\Sigma}\varphi^{\frac{2n}{n-2}}d\sigma)+b(\int_{\partial\Sigma}\varphi^{\frac{2(n-1)}{n-2}}d\sigma_{\partial \Sigma})^{\frac{n}{n-1}}\big)^{\frac{n-2}{n}}},
\end{equation}
where $k_g$ denotes the mean curvature of $\partial \Sigma$, $R_g$ is the scalar
curvature of $\Sigma$, $d\sigma$ and
$d\sigma_{\partial \Sigma}$ denote the volume element of $\Sigma$ and the area element of $\partial \Sigma$, respectively.

The Yamabe constant of $(\Sigma,g)$ is defined by
\begin{equation}\label{yamainv}
\displaystyle Q_g^{a,b}(\Sigma,\partial \Sigma)=\inf_{\varphi\in C^{\infty}(\Sigma,\Bbb{R}^{+})}Q_g^{a,b}(\varphi):(a,b)\in\{(0,1), (1,0)\},
\end{equation}
which is invariant under conformal change of the metric $g$ (see \cite{E}, \cite{E2}).
It is not difficult to verify that $-\infty\leqslant Q_g^{1,0}(\Sigma,\partial \Sigma)\leqslant Q_g^{1,0}(\mathbb{S}^n_+,\partial \mathbb{S}^n_+)$,
where $\mathbb{S}^n_+$ denotes the standard half sphere  and $-\infty\leqslant Q_g^{0,1}(\Sigma,\partial \Sigma)\leqslant Q_g^{0,1}(B^n,\partial B^n)$,
where $B^n$ is the unit ball in $\mathbb{R}^n$ equipped with the canonical metric.

Let $[g]$ and $\mathcal{C}(\Sigma)$ denote the conformal class of a Riemannian metric $g$ and
the space of all conformal classes on $\Sigma$, respectively. We may then define the \textit{Yamabe invariant} of a
compact manifold $\Sigma$ with boundary $\partial\Sigma$ by taking the supremum of the Yamabe constants
over all conformal classes
\begin{equation}\label{def}
\sigma^{a,b}(\Sigma,\partial\Sigma)=\sup_{[g]\in\mathcal{C}(\Sigma)}\inf_{\varphi>0}Q_g^{a,b}(\varphi).
\end{equation}
Schwartz \cite{S} showed that this invariant
is monotonic when attaching a handle over the boundary. As consequence, for example, a handlebody $\mathcal{H}^n$ has maximal invariant, i.e.,
$\sigma^{a,b}(\mathcal{H},\partial \mathcal{H})=\sigma^{a,b}(\mathbb{S}_+^n,\partial \mathbb{S}_+^n)$ for $(a,b)\in\{(0,1), (1,0)\}$.

In a two dimensional Riemannian surface the mean curvature of the boundary coincides with the geodesic curvature of the boundary. Therefore, the
Gauss-Bonnet Theorem implies that the Yamabe invariant of a compact surface $\Sigma$ with boundary $\partial\Sigma$ is given by
a multiple of the Euler characteristic $4\pi\chi(\Sigma)$, where $\chi(\Sigma)$ depends on the genus and
on the number of boundary components of $\Sigma$. In fact, in a certain sense, the Yamabe invariant can be viewed as generalization of the Euler characteristic in higher dimensions.

Let $M^n$ be a Riemannian manifold with boundary $\partial M$. Assume that $M$ contains a properly
embedded hypersurface $\Sigma$ with boundary $\partial\Sigma$. Let $R^M$ and $H^{\partial M}$ denote the scalar curvature of $M$ and
the mean curvature of $\partial M$, respectively. In this work, we let $vol(\Sigma)$ denote the volume ($(n-1)$-dimensional
Hausdorff measure) of $\Sigma$ while $Area(\partial\Sigma)$
denotes the area ($(n-2)$-dimensional Hausdorff measure) of its boundary $\partial\Sigma$, both with respect to the induced metric.

In \cite{SZ}, Shen and Zhu obtained some estimates on the area of compact stable minimal surfaces in three-manifolds
with bounds on the scalar curvature. Moreover, Chen, Fraser and Pang \cite{CFP} obtained the same to the nonempty boundary case and low index.
In the same spirit, in Section \ref{upper}, we obtain some estimates to the volume and area of the boundary
of minimal stable free boundary hypersurfaces in terms either of the scalar curvature or the mean convexity of
the boundary of the ambient manifold. 
Recall that a manifold $M$ is mean convex 
if its boundary $\partial M$ has nonnegative mean curvature everywhere with respect to the outward normal. We have the following theorem:

\begin{Theorem}\label{estimate}
Let $M^{n}$ be a Riemannian manifold ($n\geq4$) with nonempty boundary. Assume that $M$ contains a two-sided
compact properly immersed stable minimal free boundary hypersurface $\Sigma^{n-1}$ whose induced metric is denoted by $g$.
\begin{itemize}
\item [i)] Suppose that $M$ has mean convex boundary and $\inf R^M<0$. Then, if $\sigma^{1,0}(\Sigma,\partial\Sigma)<0$,
the volume of $\Sigma$ satisfies
$$vol(\Sigma)^{\frac{2}{n-1}}\geq \frac{Q_g^{1,0}(\Sigma,\partial\Sigma)}{\inf R^M} \geq\frac{\sigma^{1,0}(\Sigma,\partial\Sigma)}{\inf R^M}.$$

\item [ii)]Suppose that $M$ has nonnegative scalar curvature and $\inf H^{\partial M}<0$.
Then, if $\sigma^{0,1}(\Sigma,\partial\Sigma)<0$, the area of $\partial\Sigma$ satisfies
$$Area(\partial\Sigma)^{\frac{1}{n-2}}\geq \frac{1}{2}\Big(\frac{Q_g^{0,1}(\Sigma,\partial\Sigma)}{\inf H^{\partial M}}\Big)\geq\frac{1}{2}\Big(\frac{\sigma^{0,1}(\Sigma,\partial\Sigma)}{\inf H^{\partial M}}\Big).$$
\end{itemize}
\end{Theorem}

The above inequalities are a consequence of the second variation of the volume as well as the definition
of the Yamabe invariant on manifold with boundary. Although volume estimates are interesting in itself,
the estimate given in item i) plays an important role in one of our rigidity results.

 We will now establish the following local rigidity result.

 \begin{Theorem}\label{princ}
Let $M^{n}$ be a Riemannian manifold ($n\geq4$) with mean convex boundary $\partial M$ such that $R^M$
is bounded from below. Let $\Sigma^{n-1}$ be a two-sided, compact, properly embedded, free boundary hypersurface which is locally volume-minimizing.

\begin{itemize}
 \item [I)]\label{ite1}If  $\inf R^M< 0$ and  $\sigma^{1,0}(\Sigma,\partial\Sigma)<0$, then
 \begin{equation}\label{ineq}
vol(\Sigma)\geq\Big(\frac{\sigma^{1,0}(\Sigma,\partial\Sigma)}{\inf R^M}\Big)^{\frac{n-1}{2}}.
\end{equation}
 Moreover, if equality holds, then in  a neighborhood of $\Sigma,\, M$ is isometric to the product
$(-\varepsilon,\varepsilon)\times \Sigma$
for some $\varepsilon>0$, with the product metric $dt^2+g$, where $g$ is the induced metric on $\Sigma$ which is Einstein such that the scalar curvature
is negative (in fact, equal to $\inf R^M$) and $\partial\Sigma$ is a minimal hypersurface with respect to the induced metric.

\item [II)]\label{ite2}If  $R^M\geqslant0$ and  $\sigma^{1,0}(\Sigma,\partial\Sigma)\leqslant0$, then in a neighborhood of $\Sigma,\, M$
is isometric to the product metric $dt^2+g$ in $(-\varepsilon,\varepsilon)\times \Sigma$ for some $\varepsilon>0$,
where $g$ is the induced metric on $\Sigma$ that is Ricci flat and $\partial\Sigma$
is a minimal hypersurface in the induced metric.
\end{itemize}

\end{Theorem}

The proof of the local splitting relies on a construction of a one-parameter family of
properly embedded free boundary hypersurfaces with constant mean curvature. 
This, together with the resolution of the Yamabe problem for compact manifolds with
boundary, implies that each hypersurface has the same volume. For this volume comparison, we adapt a technique developed by Moraru \cite{M}.
After, we exhibited an isometry from $(-\varepsilon,\varepsilon)\times \Sigma$
into a neighborhood of $\Sigma$.

In view of the result in Theorem \ref{princ}, it is interesting to know what happens when, in higher dimension, 
the ambient manifold has scalar curvature bounded below by
a positive constant and mean convex boundary. We observe that an estimate as (\ref{ineq}) cannot occur. 
For example, let $M:= \Sigma\times\mathbb{R}$ be a manifold equipped with the product metric, where 
$\Sigma=\mathbb{S}^{n-2}_+\times \mathbb{S}^1(r)$ and $\mathbb{S}^1(r)$ is the circle of positive radius $r$.
 Note that $M$ has positive scalar curvature and nonnegative
mean curvature of the boundary, while the volume of $\Sigma$ is arbitrarily large when $r$ increases.

Now, consider a variation of $\Sigma$ given by smooth mappings  $\mathfrak{f}:(-\varepsilon,\varepsilon)\times\Sigma\rightarrow M $, $\varepsilon>0$,
such that $\mathfrak{f}(t,\cdot)$ is an embedding and $\mathfrak{f}(t,\partial\Sigma)$ is contained in $\partial M$ for all $t\in(-\varepsilon,\varepsilon)$.
 We can associate to $\mathfrak{f}$ a function $\mathcal{V}: (-\varepsilon,\varepsilon)\rightarrow\mathbb {R}$ defined by
 \begin{equation}\label{vol}
\mathcal{V}(t)=\int_{[0,t]\times\Sigma}\mathfrak{f}^*dV,  
 \end{equation}
that measure the signed $n$-dimensional volume enclosed between $\mathfrak{f}(0,\cdot)$ and $\mathfrak{f}(t,\cdot)$.

We state our next rigid result that gives a small contribution to the theory proving appropriate extensions of the result contained 
in \cite[Theorem 3.2]{Y}.

\begin{Theorem}\label{quart}
Let $M$ be a Riemannian $n$-dimensional manifold with scalar curvature $R^M\geq -n(n-1)$ and mean convex boundary. Assume that $M$
contains a two-sided, compact, properly embedded, free boundary hypersurface $\Sigma$ such that $\sigma^{1,0}(\Sigma,\partial\Sigma)\leq0$.
If $\Sigma$ locally minimizes the functional $vol(\Sigma)-(n-1)\mathcal{V}(0),$ then $\Sigma$ has a neighborhood in $M$ which is isometric to
$(-\varepsilon,\varepsilon)\times \Sigma$ with the metric $dt^2+e^{2t}g$ Ricci flat,
where $g$ is the induced metric on $\Sigma$ and $\partial\Sigma$ is a minimal hypersurface with respect to the induced metric by $g$.
\end{Theorem}

\begin{remark}
In \cite{E1}, Escobar showed that $Q_g^{1,0}(\Sigma,\partial\Sigma)$ is positive (zero, negative)
if and only if $Q_g^{0,1}(\Sigma,\partial\Sigma)$ is positive (zero, negative). Therefore, we can prove item II) in Theorem \ref{princ}
and Theorem \ref{quart} by changing $\sigma^{1,0}(\Sigma,\partial\Sigma)$ by $\sigma^{0,1}(\Sigma,\partial\Sigma)$.
\end{remark}

The outline of the paper is as follows: In Section \ref{prelim}, 
we recall facts about the Yamabe problem with boundary and the stability of free boundary hypersurfaces. In Section \ref{upper}, 
we give some estimates on the volume and area of the boundary of minimal stable free boundary hypersurfaces
which will also help in the proofs of the main results. Finally, in Section \ref{four} and \ref{five} we prove Theorems \ref{princ} and \ref{quart}.

\begin{acknowledgement}
 The authors would like to thank Levi Lima for valuable discussions and
Lucas Ambrozio for many useful comments on an earlier version of this paper.
\end{acknowledgement}

\section {Preliminaries and Notations}\label{prelim}

In this section, we shall review known results about minimal free boundary hypersurfaces and their stability and some
fundamental facts about the Yamabe problem with boundary that play a fundamental role in the proof of our theorems.

\subsection{Free boundary condition and stability}
Let $M$ be a compact Riemannian manifold with nonempty boundary $\partial M $ and consider $\mathfrak{f}:\Sigma\rightarrow M$ a
compact hypersurface properly immersed, i.e., $\mathfrak{f}$ is an immersion and
$\mathfrak{f}(\Sigma)\cap\partial M =\mathfrak{f}(\partial\Sigma)$. We say that $\Sigma$ is free boundary if $\Sigma$ meets $\partial M $ orthogonally
along $\partial\Sigma$. Thus, if $X$ is a normal vector field along $\Sigma$ we have that $X$
is tangent to $\partial M $ along $\partial\Sigma$.

Suppose $\Sigma$ is two-sided in the sense that carries a smooth unit normal vector field $N$ globally defined on $\Sigma$. This allows us to consider 
that any normal vector field can be written as $X=\varphi N$, where $\varphi\in C^{\infty}(\Sigma)$.
We now consider a one-parameter family of properly immersed hypersurfaces
$\mathfrak{f}(t,\cdot):\Sigma\rightarrow M$ for $t\in(-\varepsilon,\varepsilon)$ with initial velocity
$$
\frac{\partial}{\partial t} \mathfrak{f}(t,\cdot)\big|_{t=0}=X,
$$
where $\mathfrak{f}((-\varepsilon,\varepsilon)\times \partial \Sigma)$ is contained in $\partial M$ and $\mathfrak{f}(0,\cdot) = \mathfrak{f}$.
From now on we assume that $\mathfrak{f}(t,\Sigma)=\Sigma_t$. A well known standard computation gives
the first variation formula of volume
\begin{equation}\label{variaare}
\delta\Sigma(\varphi)=\int_{\Sigma}H\varphi d\sigma +\int_{\partial\Sigma}\langle X, \nu \rangle d\sigma_{\partial \Sigma},
\end{equation}where $H$ is the mean curvature of $\Sigma$ in $M$. From (\ref{variaare}), we note that $\Sigma$ is a
critical point to the variational problem if and only if $\Sigma$ is minimal and $\Sigma$ is free boundary.

Next, we define the following function $\mathcal {J}: (-\varepsilon,\varepsilon)\rightarrow\mathbb {R}$ given by
$\mathcal{J}(t)=vol(\Sigma_t)-(n-1) \mathcal{V}(t)$. By taking into account that 
$$
\mathcal{V}'(0)=\int_{\Sigma}\varphi d\sigma,  
$$
which was proved in \cite[Lemma 2.1(ii)]{BCE}, we obtain
\begin{equation}\label{primvaria}
\mathcal{J}'(0)=\int_{\Sigma}(H-(n-1))\varphi d\sigma +\int_{\partial\Sigma}\langle X, \nu \rangle d\sigma_{\partial \Sigma}.
\end{equation}
Analogously, we have that $\Sigma$ is critical to this variational problem if and only if $H=n-1$ and $\Sigma$ is free boundary.

The Jacobi operator, sometimes called stability operator, is a second order linear operator $\mathcal{L}$ given by
$\mathcal{L}=\Delta_{\Sigma}+Ric(N,N)+\|h^{\Sigma}\|^2$, where $\Delta_{\Sigma}$ is the Laplacian, $Ric$ is the Ricci curvature of $M$
and $h^\Sigma$ is the second fundamental form of $\Sigma$ with respect to the unit normal $N$.

The second variation of volume of a minimal hypersurface is given by
\begin{eqnarray}\label{secvar}
\delta^2\Sigma(\varphi,\varphi)=-\int_{\Sigma}\varphi\mathcal{L}\varphi d\sigma+\int_{\partial \Sigma}(\frac{\partial\varphi}{\partial \nu}-\Pi(N,N)\varphi )\varphi d\sigma_{\partial \Sigma},
\end{eqnarray}
where $\Pi$ denotes the second fundamental
form of $\partial M$ with respect to the inner unit normal vector. Assuming now that $\Sigma$ is critical, the second variation formula
$\mathcal{J}''(0)$ coincides with $\delta^2\Sigma(\varphi,\varphi)$ (see \cite[Proposition 3.5]{CR}).

We recall that the index of a minimal hypersurface is defined as the maxima dimension of any subspace of
$C^{\infty}(\Sigma)$ on which $\delta^2\Sigma(\varphi,\varphi)$ is negative definite. Roughly speaking, it measures the number of
independent directions in which the hypersurface fails to minimize volume. We say that a minimal two-sided
hypersurface $\Sigma$ is stable if and only if $\delta^2\Sigma(\varphi,\varphi)\geqslant0$ for all $\varphi\in C^{\infty}(\Sigma)$ or,
equivalently, the index of $\Sigma$ is equal to zero.
Note that if a hypersurface $\Sigma$ is locally volume-minimizing, then $\Sigma$ is stable and minimal. Similarly,
$\Sigma$ is $\mathcal{J}$-stable when $\mathcal{J}''(0)\geq0$ for all normal variation of $\Sigma$.

\subsection{Basic facts about the Yamabe problem on manifolds with boundary}

Consider a compact $(n-1)$-dimensional Riemannian manifold $(\Sigma,g)$ with nonempty boundary $\partial \Sigma$.
The Yamabe problem asserts that any Riemannian metric on a closed manifold is conformal
to a metric with constant scalar curvature. This problem was completely solved after works of Trundiger \cite{Tr}, Aubin \cite{Au}
and Schoen \cite{Sc}. There are two ways to extend this problem to manifolds with boundary,
the first is to find a metric $\tilde{g}=\varphi^{\frac{4}{n-3}}g$, conformally related to $g$, such that $R_{\tilde{g}}$ is constant equal to $C\in \mathbb{R}$ and
$\kappa_{\tilde{g}}$ is zero which is equivalent to the existence of a critical point of the functional $Q_g^{1,0}(\varphi)$ for all smooth
positive function $\varphi$ on $\Sigma$ satisfying
\begin{equation}\label{minimi0}
\left\{\begin{matrix}
\Delta_{\Sigma}\varphi-\frac{n-3}{4(n-2)}R_g\varphi+\frac{n-3}{4(n-2)}C\varphi^\frac{n+1}{(n-3)}=0\ & \textrm{in $\Sigma$}\\
\frac{\partial \varphi}{\partial \nu}+\frac{n-3}{2(n-2)}\kappa_g\varphi=0 & \textrm{on $\partial\Sigma$,}
\end{matrix}\right.
\end{equation}
where $\nu$ is the outward normal vector to $\partial\Sigma$. The second natural extension is to find a conformal scalar-flat metric
on $\Sigma$ which has as boundary a hypersurface with constant mean curvature equal to $K\in \mathbb{R}$ which corresponds, in analytical terms,
to find a positive solution of
\begin{equation}\label{minimi}
\left\{\begin{matrix}
\Delta_{\Sigma}\varphi-\frac{n-3}{4(n-2)}R_g\varphi=0 & \textrm{in $\Sigma$}\\
\frac{\partial \varphi}{\partial \nu}+\frac{n-3}{2(n-2)}k_g\varphi=\frac{n-3}{2(n-2)}K\varphi^{\frac{n-1}{n-3}} & \textrm{on $\partial\Sigma$.}
\end{matrix}\right.
\end{equation}

Since $Q_g^{1,0}$ and $Q_g^{0,1}$ do not satisfy the Palais-Smale condition, standard variational methods
cannot guarantee the existence of minimizers. Escobar \cite{E} proved that if
$Q_g^{1,0}(\Sigma,\partial\Sigma)<Q_g^{1,0}(\mathbb{S}^{n-1}_+,\partial \mathbb{S}^{n-1}_+)$,
then there exists a minimizing solution to problem (\ref{minimi0}), where $C$ has the same sign as $Q_g^{1,0}(\Sigma,\partial \Sigma)$.
When $Q_g^{0,1}(\Sigma,\partial\Sigma)$ is finite and
$Q_g^{0,1}(\Sigma,\partial\Sigma)<Q_g^{0,1}(B^{n-1},\partial B^{n-1})$,
there exists a smooth metric of flat scalar curvature and mean curvature on the boundary equal to $K$ that
has the same sign as  $Q_g^{0,1}(\Sigma,\partial\Sigma)$.
There are a lot of interesting papers related with this subject,
we indicate for instance \cite{E}, \cite{E1}, \cite{E2}, \cite{M1}, \cite{M} and \cite{Al}.

\section{Lower bounds to the volume of $\Sigma$ and the area of $\partial\Sigma$}\label{upper}

\begin{proof}[Proof of Theorem \ref{estimate}]
Initially, we observe that Gauss equation implies
\begin{equation}\label{gauss}
Ric(N,N)=\frac{1}{2}(R^M-R_g+ H^2-\|h^{\Sigma}\|^2).
\end{equation}

We use (\ref{gauss}) in the stability condition to obtain
\begin{equation}\label{mud2}
\int_{\Sigma}\Big(2\|\nabla \varphi\|_g^2+(R_g-R^M-\|h^{\Sigma}\|^2)\varphi^2\Big)d\sigma-2\int_{\partial\Sigma}\Pi(N,N)\varphi^2d\sigma_{\partial\Sigma}\geq0.
\end{equation}

On the other hand, since $\Sigma$ meets $\partial M$
orthogonally along $\partial\Sigma$, we have that the unit conormal vector $\nu$ of $\partial\Sigma$ that points outside $\Sigma$
coincides with the inner unit normal $Z$ of $\partial M$ that points outside $M$. Therefore  
$$
\kappa_g=\sum_{i=1}^{n-2}\langle\nabla_{e_i}\nu,e_i\rangle=\sum_{i=1}^{n-2}\langle\nabla_{e_i}Z,e_i\rangle,
$$
where $\{e_1,\ldots,e_{n-2}\}$ is
an orthonormal basis for $T\partial\Sigma$. Thus, we obtain
\begin{equation}\label{free}
\Pi(N,N)=H^{\partial M}-\kappa_g \textrm{\;\;\;\;along  $\partial \Sigma$},
\end{equation}
where $H^{\partial M}$ is the mean curvature of $\partial M$ with respect to the inner unit conormal.

By using that $a_n=\frac{4(n-2)}{n-3}>2$ for all $n\geq4$ and (\ref{free}), it follows that

\begin{eqnarray}\label{desi}
0&\leq &\int_{\Sigma}\Big(a_n\|\nabla \varphi\|_g^2+R_g\varphi^2\Big)d\sigma-\int_{\Sigma}R^M\varphi^2d\sigma\nonumber\\
&-&\int_{\partial\Sigma}2H^{\partial M}\varphi^2d\sigma_{\partial\Sigma}+\int_{\partial\Sigma}2\kappa_g\varphi^2d\sigma_{\partial\Sigma}.
\end{eqnarray}

Next we use H\"older's inequality to deduce $$\inf R^M\int_{\Sigma}\varphi^2d\sigma\geq\inf R^Mvol(\Sigma)^{\frac{2}{n-1}}\Big(\int_{\Sigma}\varphi^{\frac{2(n-1)}{n-3}}d\sigma\Big)^{\frac{n-3}{n-1}}.$$
Whence we have
\begin{eqnarray}
\inf R^Mvol(\Sigma)^{\frac{2}{n-2}}\Big(\int_{\Sigma}\varphi^{\frac{2(n-1)}{n-3}}d\sigma\Big)^{\frac{n-3}{n-1}}&\leq&\int_{\Sigma}\big(a_n\|\nabla \varphi\|_g^2+R_g\varphi^2\big)d\sigma\nonumber\\
&+&\int_{\partial\Sigma}2\kappa_g\varphi^2d\sigma_{\partial\Sigma}.\nonumber
\end{eqnarray}

Thus for a smooth positive function $\varphi$, we obtain

\begin{equation}\label{yama2}
\inf R^Mvol(\Sigma)^{\frac{2}{n-2}}\leq\frac{\int_{\Sigma}\big(a_n\|\nabla\varphi\|_g^2+R_g\varphi^2\big)d\sigma +2\int_{\partial \Sigma}\kappa_g\varphi^2d\sigma_{\partial M}}{\Big(\int_{\Sigma}\varphi^{\frac{2(n-1)}{n-3}}d\sigma\Big)^{\frac{n-3}{n-1}}}.
\end{equation}

 Using the definition of the Yamabe constant (\ref{yamainv}) into (\ref{yama2}), we have
\begin{eqnarray}\nonumber
\inf R^Mvol(\Sigma)^{\frac{2}{n-2}}&\leq& Q_g^{1,0}(\Sigma,\partial\Sigma)\leq \sigma^{1,0}(\Sigma,\partial\Sigma),\nonumber
\end{eqnarray}
where we used the definition of Yamabe's invariant (\ref{def}). Thus we complete the proof of our first item.

Reasoning as in the previous case, for a smooth positive function $\varphi$, we obtain 

\begin{equation}\label{yama3}
2\inf H^{\partial M}Area(\partial\Sigma)^{\frac{1}{n-2}}\leq\frac{\int_{\Sigma}\big(a_n\|\nabla\varphi\|_g^2+R_g\varphi^2\big)d\sigma +2\int_{\partial \Sigma}\kappa_g\varphi^2d\sigma_{\partial M}}{\Big(\int_{\partial\Sigma}\varphi^{\frac{2(n-2)}{n-3}}d\sigma\Big)^{\frac{n-3}{n-2}}}.
\end{equation}

Therefore, we get
\begin{eqnarray}\nonumber
\inf H^{\partial M}Area(\partial\Sigma)^{\frac{1}{n-2}} &\leq& \frac{1}{2}Q_g^{0,1}(\Sigma,\partial\Sigma)\leq \frac{1}{2}\sigma^{0,1}(\Sigma,\partial\Sigma),\nonumber
\end{eqnarray}which completes the proof of the theorem.
\end{proof}

\section{Proof of Theorem \ref{princ}}\label{four}

\subsection{Infinitesimal rigidity}
By previous section we obtain inequality (\ref{ineq}). Now, our next goal is to construct a CMC
foliation of free boundary hypersurfaces. Before, we prove the following result.

\begin{proposition}\label{propig}
Suppose that in $\Sigma$ inequality (\ref{ineq}) becomes equality. Then $\Sigma$ is totally geodesic, $R^M=\inf R^M$ and $Ric(N,N)=0$ along $\Sigma$, while
$H^{\partial M}=0$ and $\Pi(N,N)=0$ along $\partial\Sigma$, the boundary $\partial\Sigma$ is 
a minimal hypersurface with respect to the induced metric. Moreover, the induced metric on $\Sigma$ is Einstein.
\end{proposition}

\begin{proof}
From the resolution of Yamabe's problem, there exists $\varphi_{min}>0$
for which the infimum in $Q_g^{1,0}(\Sigma,\partial \Sigma)$ is achieved.
If inequality (\ref{ineq}) becomes equality in $\Sigma$, then
it follows that all inequalities in the proof of Theorem \ref{estimate}  are in fact equalities. 

Firstly, note that $\Sigma$ is totally geodesic
and $H^{\partial M}=0$ along $\partial \Sigma$.
Since we use the strict inequality $a_n-2>0$ to obtain (\ref{desi}) we have that $\|\nabla \varphi_{min}\|_g^2=0$
which implies that $\varphi_{min}$ is constant.

Consider the following Robin-type boundary value problem:
\begin{equation}\label{rob}
\left\{\begin{matrix}
-\mathcal{L}\phi= \lambda\phi& \textrm{in $\Sigma$}\\
\frac{\partial}{\partial \nu}\phi=\Pi(N,N)\phi & \textrm{on $\partial\Sigma.$}
\end{matrix}\right.
\end{equation}
Let $\lambda_1$ be the first eigenvalue of the above problem. It is well known that
$$\lambda_1=\inf_{\int_{\Sigma}\phi^2=1}\Big(\int_{\Sigma}\big(\|\nabla\phi\|_g^2+(Ric(N,N)+\|h^{\Sigma}\|^2)\phi^2\big)d\sigma +\int_{\partial \Sigma}\Pi(N,N)\phi^2d\sigma_
{\partial \Sigma}\Big).$$

It follows from $\delta^2\Sigma(\varphi_{min},\varphi_{min})=0$ that $\lambda_1=0$. Thus, the constant functions satisfy (\ref{rob}) and
we obtain that $\Pi(N,N)=0$ and $Ric(N,N)=0$. Note also that equality in (\ref{mud2}) implies $R^M=\inf R^M$ along $\Sigma$.
It remains to prove that $\Sigma$ carries an Einstein metric. In fact, given any smooth symmetric (0,2)-tensor $h$, 
we define a family of Riemannian metrics $\{g(r)\}_{r\in(-\epsilon,\epsilon)}$, where $g(r)=g+rh$.
 From the resolution of the Yamabe problem on manifold with boundary there
 exists a unique positive function  $u_r>0$ such that  $\tilde{g}(r)=u_r^{\frac{4}{n-3}}g(r)$ has constant scalar curvature equal
to $Q_{g(r)}^{1,0}(\Sigma,\partial\Sigma)<0$
 and zero mean curvature on the boundary for all $r\in(-\epsilon,\epsilon)$. 
  Note that since the Yamabe invariant does not
depend on $r$, we have $Q_{g(r)}^{1,0}(\Sigma,\partial\Sigma)\leq\sigma^{1,0}(\Sigma,\partial\Sigma)$ for all $r\in(-\epsilon,\epsilon)$, i.e.,
$\frac{\partial}{\partial r}Q_{g(r)}^{1,0}(\Sigma,\partial\Sigma)$ equal to zero at $r=0$ provided $\sigma^{1,0}(\Sigma,\partial\Sigma)$ is maximum for
$Q_{g(r)}^{1,0}(\Sigma,\partial\Sigma)$ as a function of $r$.

On the other hand, it is well known that the derivative of the scalar curvature satisfies

$$\frac{\partial}{\partial r}R_{\tilde{g}(r)}\Big|_{r=0}=\textrm{div}(\textrm{div} h-dtr_{\tilde{g}(0)}h)-\langle Ric^{\Sigma},h\rangle,$$
where $Ric^{\Sigma}$ denotes the Ricci curvature on $\Sigma$.

We notice that since
 $\tilde{g}(0)$ and $g$ are in the same conformal class with the same scalar curvature and mean curvature on the boundary
 up to scaling, the uniqueness of Yamabe's problem for manifold with boundary for $inf R^M<0$ and
the  boundary being minimal imply that $\tilde{g}(0)=g.$

 Let $vol(\Sigma,\tilde{g}(r))$ denote the volume of $\Sigma$ in the metric $\tilde{g}(r)$
for $r\in(-\epsilon,\epsilon)$, so we have

\begin{eqnarray*}\nonumber
\frac{\partial}{\partial r}Q_{g(r)}^{1,0}(\Sigma,\partial\Sigma)\Big|_{r=0}&=&\frac{d}{dr}\Big( vol(\Sigma,\tilde{g}(r))^{\frac{2-n}{n}}\Big[\int_{\Sigma}R_{\tilde{g}(r)}d\sigma+
2\int_{\partial\Sigma}\kappa_{\tilde{g}(r)}d\sigma_{\partial\Sigma}\Big]\Big)\Big|_{r=0} \\
&=& vol(\Sigma)^{\frac{2-n}{n}}\Big(\frac{2-n}{n}vol(\Sigma)^{-1}\int_{\Sigma}\frac{1}{2}(tr_gh)d\sigma\int_{\Sigma}R_gd\sigma_g\Big)\nonumber\\
& +&vol(\Sigma)^{\frac{2-n}{n}}\Big(\int_{\Sigma}\langle -Ric^{\Sigma}+\frac{R_g}{2}g,h\rangle d\sigma+\int_{\Sigma}\Delta_{\Sigma}(tr_g(h))d\sigma \Big)\nonumber\\
&+ &vol(\Sigma)^{\frac{2-n}{n}}\int_{\partial\Sigma}\langle div(h),\nu\rangle,\nonumber\\
\end{eqnarray*}
where we use Stokes' Theorem and $\frac{\partial}{\partial r}\int_{\Sigma}d\sigma_r\Big|_{r=0}=\frac{1}{2}\int_{\Sigma}\langle h,g\rangle d\sigma$. Therefore,
we have

\begin{eqnarray}\nonumber
\frac{\partial}{\partial r}Q_{g(r)}^{1,0}(\Sigma,\partial\Sigma)\Big|_{r=0}&=& -vol(\Sigma)^{\frac{2-n}{n}} \int_{\Sigma}\langle Ric^{\Sigma}-\frac{1}{2}R_g g+\frac{n-2}{2n}\overline{R}g,h\rangle d\sigma\nonumber\\
& +&vol(\Sigma)^{\frac{2-n}{n}}\Big(\int_{\Sigma}\Delta_{\Sigma}(tr_g(h))d\sigma+\int_{\partial\Sigma}\langle div(h),\nu\rangle d\sigma_{\partial\Sigma}\Big), \nonumber
\end{eqnarray}
 where $\overline{R}$ denotes the average scalar curvature $\overline{R}=vol(\Sigma)^{-1}\int_{\Sigma}R_gd\sigma$.

Thanks to identity (\ref{gauss}), the scalar curvature of $\Sigma$ is constant with respect to the induced metric, so we obtain
$$-\int_{\Sigma}\langle Ric^{\Sigma}-\frac{1}{n}R_g g,h\rangle d\sigma+\int_{\partial\Sigma}\langle div(h),\nu\rangle d\sigma_{\partial\Sigma}+\int_{\Sigma}\Delta_{\Sigma}(tr_g(h))d\sigma=0.$$

Choosing $h$ as the traceless Ricci tensor, we derive the following expression
$$-\int_{\Sigma} \|Ric^{\Sigma}-\frac{1}{n}R_g g\|^2d\sigma+\frac{n-2}{2n}\int_{\partial\Sigma}\langle \nabla R, \nu\rangle d\sigma_{\partial\Sigma}=0.$$

Since $\int_{\partial\Sigma}\langle \nabla R, \nu\rangle d\sigma_{\partial\Sigma}=\int_{\Sigma}\Delta_{\Sigma} R d\sigma=0,$ we deduce that the traceless Ricci tensor must vanish implying that $\Sigma$ carries an Einstein metric.
\end{proof}

\begin{proposition}\label{propig2}
 Under the considerations of item II) in Theorem \ref{princ}, we have that $\|h^{\Sigma}\|^2=0$, $R^M=0$ and $Ric(N,N)=0$ along $\Sigma$, while
 $H^{\partial M}=0$ and $\Pi(N,N)=0$ along $\partial\Sigma$,
the mean curvature of $\partial\Sigma$ in $\Sigma$ is equal to zero. Moreover,  $\sigma^{1,0}(\Sigma,\partial\Sigma)=0$ and $\Sigma$ is Ricci flat with respect to the induced metric.
\end{proposition}
\begin{proof}
Arguing as in the proof of Theorem \ref{estimate}, we have $$
0=\inf R^Mvol(\Sigma)^\frac{2}{n-1}\leq\sigma^{1,0}(\Sigma,\partial\Sigma)\leq0,
$$
then the above inequalities become equalities, hence the proof of Proposition \ref{propig2} follows in the same steps like that one of Proposition \ref{propig}.
Moreover, by using that the induced metric on $\Sigma$ is Einstein as well as (\ref{gauss}) we deduce that $\Sigma$ is Ricci flat.
\end{proof}

To conclude this section we recall that a two-sided properly embedded free boundary hypersurface $\Sigma$ in
$M$ is called by \textit{infinitesimally rigid}, if
$\Sigma$ is totally geodesic,
$R^M=\inf R^M$ and $Ric(N,N)$ vanishes along $\Sigma$,
the mean curvature of $\partial M$ is constant equal to $\inf H^{\partial M}$ at every point of $\partial\Sigma$
and the induced metric on $\Sigma$ is Einstein (i.e.,
the induced metric on $\Sigma$ attains the Yamabe invariant). We also remark that basic examples of such
manifolds are horizontal slices $\{r\}\times\Sigma$ in
a Riemannian manifold $\mathbb{R}\times \Sigma$ endowed with the product metric,
where $\Sigma$ is an Einstein manifold with constant scalar curvature and boundary being a hypersurface with constant mean curvature.

\subsection{Local foliation by CMC free boundary hypersurfaces}

When $\Sigma$ is infinitesimally rigid allows us to use the Implicit Function Theorem to obtain a foliation in
a neighborhood of $\Sigma$ by constant mean curvature free boundary hypersurfaces. This is contained in the next proposition
that was inspired by the work of Bray, Brendle and Neves \cite{BBN} whose proof is a slight modification of
that one presented in Ambrozio \cite{A} or in Nunes \cite{N}.

Considering a properly embedded infinitesimally rigid hypersurface $\Sigma$ in $M$, we obtain a vector
field $Y$ in $M$ that coincides with $N$ in $\Sigma$ and $Y(p)$ is tangent to $\partial M $ for all $p\in\partial M$.
Let $\psi =\psi(t, x)$ denote the flow of $Y$.

\begin{proposition}[CMC Foliation]\label{folhea}
Let $M^n$ be a Riemannian manifold with nonempty boundary. Assume that $M$ contains a properly embedded
free boundary hypersurface $\Sigma$ such that $H^{\partial M}$ and $R^M$ are bounded from below. If $\Sigma$ is infinitesimally rigid, then there exist
$\varepsilon>0$ and a smooth function $\mu:(-\varepsilon,\varepsilon)\times\Sigma\rightarrow\mathbb{R}$ such that $\Sigma_{t}:=\{\psi(\mu(t,x)+t,x),x\in\Sigma\}$
is a family of compact free boundary hypersurfaces with constant mean curvature. In addition $\mu(0,x)=0$, $\frac{\partial \mu}{\partial t} (0,x)=0$
and $\int_{\Sigma}\mu(t,\cdot)d\sigma=0$ for each $x\in\Sigma$ and $t\in(-\varepsilon,\varepsilon).$
\end{proposition}

\begin{proof}
A CMC foliation can be constructed as in \cite{A}. Let $E_n=\{u\in C^{n,\alpha}(\Sigma); \int_{\Sigma}u=0\}$
be a Banach spaces with H\"older exponent $\alpha\in(0,1)$ for each $n\in \mathbb{N}$. Choose $\tau>0$, $\delta>0$ and a real function $u$ in the
open ball $B_{\delta}(0)=\{u\in C^{2,\alpha}(\Sigma); \|u\|_{{2,\alpha}}<\delta\}$ such that the set
$\Sigma_{u+t}=\{\psi(u(x)+t,t);x\in\Sigma\}$ defines a compact properly embedded hypersurface for all $(t,u)\in(-\tau,\tau)\times B_{\delta}(0)$.

Let $Z$ be the unit normal vector field of $\partial M$ that coincides
with the exterior conormal $\nu$ of $\partial \Sigma$. We define a mapping
$\Phi:(-\tau,\tau)\times (B_\delta(0)\cap E_2)\rightarrow E_0\times C^{1,\alpha}(\partial\Sigma)$ putting
$$\Phi(t,u)=(H(t+u)-\frac{1}{vol(\Sigma)}\int_{\Sigma}H(t+u)d\sigma,\langle N_{t+u},Z_{t+u}\rangle),$$
where $N_u$ denotes the unit normal field of $\Sigma_u,\, Z_u=Z\Big|_{\partial \Sigma_u}$ and $ H(u)$ is the mean curvature of $\Sigma_u.$
Note that $\Phi$ is well-defined and $\Phi(0,0)=(0,0)$ provided $\Sigma_0=\Sigma$ is minimal and free boundary.

Consider the mapping $f:(-\tau,\tau)\times\Sigma\rightarrow M$ so that $f(t,\cdot)=\psi(tv(\cdot),\cdot)$ which gives a variation for
each $v\in E_2$, whose variational vector field is $\frac{\partial}{\partial t}f\Big|_{t=0}=vY$ on $\Sigma$.

We compute $D\Phi_{(0,0)}(0, v)$ for each $v\in E_2$

\begin{eqnarray}\nonumber
D\Phi_{(0,0)}(0, v)&=&\frac{d\Phi}{ds}\Big|_{t=0}(0,sv)\\
&=&(-\Delta_{\Sigma}v+\frac{1}{vol(\Sigma)}\int_{\partial\Sigma}\frac{\partial v}{\partial\nu}d\sigma_{\partial\Sigma},-\frac{\partial v}{\partial\nu})\nonumber,
\end{eqnarray}
where we used that $\Sigma$ is infinitesimally rigid.

Now, choosing $w\in E_0$ and $z\in C^{1,\alpha}(\partial\Sigma)$ we deduce $$
\int_{\Sigma}\Big(w+\frac{1}{vol(\Sigma)}\int_{\partial\Sigma}zd\sigma_{\partial\Sigma}\Big)d\sigma=
\int_{\partial\Sigma}zd\sigma_{\partial\Sigma},
$$
 which implies by Theorem 2.1 of \cite{Na} that there exists a unique function $\theta\in E_2$ solving the Neumann boundary problem

\begin{equation}
\left\{\begin{matrix}
\Delta_{\Sigma}\theta = w+\frac{1}{vol(\Sigma)}\int_{\partial\Sigma}zd\sigma_{\partial\Sigma} & \textrm{in $\Sigma$}\\
\frac{\partial \theta}{\partial t}=-z & \textrm{on $\partial\Sigma$}
\end{matrix}\right..
\end{equation}
Hence, $D\Phi_{(0,0)}(0,\theta)=(w,z)$, so $D\Phi_{(0,0)}$ is an isomorphism
when restricted to $0\times E_2$ (see also \cite{L}, p. 137). Hence, we are in position to use the Implicit Function Theorem to
guarantee the existence of $\varepsilon>0$ as well as a smooth function $\mu$ such that $\mu(0,x)=0$ and $\mu(t,\cdot)\in B_\delta(0)\cap E_2$.
We can construct a variation $ G(t,x)=\psi(\mu(t,x)+t,x)$ whose velocity vector is equal to $\Big(\frac{\partial \mu}{\partial t}+1\Big)N$ on $\Sigma$.

Differentiating the following identity at $t=0$
$$(H(\mu(t,\cdot)+t)-\frac{1}{vol(\Sigma)}\int_{\Sigma}H(\mu(t,\cdot)+t)d\sigma,\langle N_{\mu(t,\cdot)+t},X_{\mu(t,\cdot)+t}\rangle)=(0,0),$$
we get that $\frac{\partial}{\partial t} \mu(0,x)$ is constant since it satisfies the homogeneous Neumann problem.
However, taking once more the derivative at $t=0$ of $\int_{\Sigma}\mu(t,\cdot)d\sigma=0$, we obtain
$$\int_{\Sigma}\frac{\partial \mu}{\partial t}(0,\cdot)d\sigma=0,$$
which implies $\frac{\partial \mu}{\partial t}(0,x)=0$.

We remark that $$\frac{\partial G}{\partial t}(0,x)=N\textrm{ \;\;\;\; for all $x\in\Sigma$,}$$
with $G(0,x)=x$. Thus, we can assume that, decreasing $\varepsilon$ if necessary, a neighborhood of $\Sigma$ is parametrized by $G$. Hence, the assertion follows and we complete the proof.

\end{proof}

We construct a foliation on a neighborhood of $\Sigma$ in $M$ by properly embedded free boundary
$\{\Sigma_t\}_{t\in(-\varepsilon,\varepsilon)}$. We consider the following mapping $\mathfrak{f}(t,\cdot):\Sigma\rightarrow M $ given by
$\mathfrak{f}(t,x) = \psi(\mu(t,x)+t,x)$ that parametrizes the foliation $\{\Sigma_t\}_{t\in(-\varepsilon,\varepsilon)}$ around
$\Sigma$ and denote by $d\sigma_t$ and $d\sigma_{\partial\Sigma_t}$
the volume element of $\Sigma_t$ and the area element  of $\partial \Sigma_t$ with respect to the induced metric by $\mathfrak{f}(t,\cdot)$, respectively.

Consider the operator $$
\mathcal{L}(t)=\Delta_{\Sigma_t} + Ric (N_t, N_t)+\|h^{\Sigma_t}\|^2,
$$
where $\Delta_{\Sigma_t}$, or just $\Delta_t$ when there is no ambiguity, stands for the Laplacian of $\Sigma_t$ in the induced metric,
$N_t$ is the unit normal vector field of $\Sigma_t$ which we assume that depends smoothly on $(-\varepsilon,\varepsilon)\times\Sigma$.
Moreover, $h^{\Sigma_t}$ denotes the second fundamental form of $\mathfrak{f}(t,\cdot)$ with respect to $N_t$.

For each $t\in(-\varepsilon,\varepsilon)$ the \textit{lapse function} $\ell_t:\Sigma\rightarrow \mathbb{R}$
is defined by
$$
\ell_t(x) = \big \langle N_t(x),X_t(x)\big\rangle,
$$
where $X_t=\frac{\partial}{\partial t}\mathfrak{f}(t,\cdot)$.
The next lemma is fundamental and its proof can be found in \cite[Proposition 18]{A}.

\begin{lemma} \label{rambr}
Let $\Sigma_t\subset M$, $t\in(-\varepsilon,\varepsilon)$, be a family of hypersurfaces of constant mean curvature free boundary hypersurface.
The lapse function $\ell_t (x)$ satisfies
\begin{eqnarray}\label{Huiskboun}
H'(t)&=&-\mathcal{L}(t)\ell_t\textrm{\;\;in $\Sigma_t$}\\
\frac{\partial \ell_t}{\partial\nu_t}&=&\Pi(N_t,N_t)\ell_t\textrm{\;\;on $\partial\Sigma_t$} ,
\end{eqnarray}
 where $H(t)$ is the mean curvature of $\Sigma_t$ and $H'=\frac{\partial}{\partial t}H$.
\end{lemma}

\subsection{Volume Comparison and Rigidity}

In order to obtain the local rigidity we need the following proposition.

\begin{proposition} \label{proimp}
Under the considerations of Theorem \ref{princ}, if $\Sigma$ is infinitesimally rigid  we have
$$
vol(\Sigma)\geq vol(\Sigma_t)
,\forall \,t\in(-\varepsilon,\varepsilon),$$
where $\{\Sigma_t\}_{t\in(-\varepsilon,\varepsilon)}$ is given as in Proposition \ref{folhea}.
\end{proposition}

\begin{proof}

Locally each $\Sigma_t$ is free boundary
with constant mean curvature which implies that the first variation formula of volume reduces to
\begin{equation}\label{RV}
\frac{d}{dt}vol (\Sigma_t)= H(t)\int_{\Sigma_t}\langle N_t,X_t\rangle d\sigma_t,
\end{equation}
for all $t\in[0,\varepsilon).$ Notice that $X_0(x)= N(x)$, so the continuity implies that, decreasing $\varepsilon$ if necessary,
$\ell_t>0$ for all $t\in(-\varepsilon,\varepsilon)$.
 If  $H(t)\leq0$ for
$t\in[0,\varepsilon)$ and $H(t)\geq0$ for $t\in(-\varepsilon,0],$ then $\frac{d}{dt}vol (\Sigma_t)\leqslant0\;\;\forall t\in[0,\varepsilon)$
 and $\frac{d}{dt}vol (\Sigma_t)\geqslant0\;\;\forall t\in[-\varepsilon,0)$. This is sufficient to settle the result. Let us show that this occurs.
 Note that using once more (\ref{gauss}) we can rewrite (\ref{Huiskboun}) as
\begin{equation}\label{eq1}
2H'(t)(\ell_t)^{-1}=-2(\ell_t)^{-1}\Delta_t\ell_t+ R_t-R^M_t-H(t)^2-\|h^{\Sigma_t}\|^2.
\end{equation}

Let $g_t$ be the induced metric on $\Sigma$. By resolution of Yamabe's problem for manifolds with boundary, there exists for each $t\in(-\varepsilon,\varepsilon)$, a metric $\widetilde{g}_t$ in the conformal class of $g_t$ having scalar curvature $\inf R^M$ and the boundary being
a minimal hypersurface. Let $u_t$ be a positive function on $\Sigma_t$ satisfying $\widetilde{g}_t=u_t^{\frac{4}{n-3}}g_t$.

Now, we will adapt the method introduced in \cite{M} to establish the volume comparison. 
First multiplying (\ref{eq1}) by $u_t^2$ and integrating along $\Sigma_t$ it becomes

\begin{equation}\label{eq2}
2\int_{\Sigma}H'(t)\frac{u_t^2}{\ell_t}d\sigma_t\leq-2\int_{\Sigma}\frac{u_t^2}{\ell_t}\Delta_t\ell_td\sigma_t+ \int_{\Sigma}R_tu_t^2d\sigma_t-\inf R^M\int_{\Sigma}u^2_td\sigma_t.
\end{equation}

By using in the left term that $\Sigma_t$ has constant mean curvature and integration by parts on the right we obtain

\begin{eqnarray}\label{eq3}
2H'(t)\int_{\Sigma}\frac{u_t^2}{\ell_t}d\sigma_t&\leq& 2\int_{\Sigma}\Big(2\frac{u_t}{\ell_t}\langle \nabla_tu_t,\nabla_t\ell_t\rangle_{g_t}
-\frac{u_t^2}{\ell_t^2}\|\nabla_t\ell_t\|_{g_t}^2\Big)d\sigma_t\nonumber\\
&-&2\int_{\partial\Sigma}\Pi(N_t,N_t)u_t^2d\sigma_{\partial\Sigma_t}+\int_{\Sigma}R_tu_t^2d\sigma_t -\inf R^M\int_{\Sigma}u^2_td\sigma_t.\nonumber\\
\nonumber
\end{eqnarray}

The Cauchy inequality with epsilon shows that

$$
2\langle\nabla_tu_t,\nabla_t\ell_t\rangle_{g_t}\leq\|\nabla_t u_t\|_{g_t}^2\epsilon(t)+\|\nabla_t\ell_t\|_{g_t}^2\frac{1}{\epsilon(t)},
$$
where $\epsilon(t)=\dfrac{\ell_t}{u_t}$.

Finally

\begin{eqnarray}\label{eqi3}
2H'(t)\int_{\Sigma}\frac{u_t^2}{\ell_t}d\sigma_t &\leq& \int_{\Sigma}(a_n\|\nabla_tu_t\|_{g_t}^2+R_tu_t^2)d\sigma_t-2\int_{\partial\Sigma}H_t^{\partial M}u_t^2d\sigma_{\partial\Sigma_t}\nonumber\\
& &+2\int_{\partial\Sigma}\kappa_tu_t^2d\sigma_{\partial\Sigma_t}-\inf R^M\int_{\Sigma}u^2_td\sigma_t,
\end{eqnarray}
where we used that $a_n>2$ for all $n\geq4$ and identity (\ref{free}).

Dividing (\ref{eqi3}) by  $(\int_{\Sigma}u_t^{\frac{2(n-1)}{n-3}}d\sigma)^{\frac{n-3}{n-1}}$, using that $\inf H^{\partial M}=0$
and setting $\Psi(t)=(\int_{\Sigma}u_t^{\frac{2(n-1)}{n-3}}d\sigma)^{\frac{3-n}{n-1}}\int_{\Sigma}\frac{u_t^2}{\ell_t}d\sigma_t$, we arrive at

\begin{eqnarray}\nonumber
2H'(t)\Psi(t)&\leq&\frac{\int_{\Sigma}\big(a_n\|\nabla_t u_t\|_{g_{t}}^2+R_t u_t^2\big)d\sigma_t +2\int_{\partial\Sigma}\kappa_tu_t^2d\sigma_{\partial\Sigma_t}}{(\int_{\Sigma}u_t^{\frac{2(n-1)}{n-3}}d\sigma)^{\frac{n-3}{n-1}}}\\
&-&\inf R^M\frac{\int_{\Sigma}u^2_td\sigma_t}{(\int_{\Sigma}u_t^{\frac{2(n-1)}{n-3}}d\sigma)^{\frac{3-n}{n-1}}}\nonumber.
\end{eqnarray}

Next we distinguish two cases:
\begin{case}
 $\inf R^M<0$ and $\sigma^{1,0}(\Sigma,\partial\Sigma)<0.$
\end{case}

It follows from the definition of the Yamabe constant (\ref{yamainv}) and H\"older's inequality that
\begin{eqnarray}\label{eq5}
2H'(t)\Psi(t)&\leq& Q^{1,0}_{g_t}(\Sigma,\partial\Sigma)-\inf R^Mvol(\Sigma_t)^{\frac{2}{n-1}}\nonumber\\
&\leq&\sigma^{1,0}(\Sigma,\partial\Sigma)-\inf R^Mvol(\Sigma_t)^{\frac{2}{n-1}},
\end{eqnarray}
where we used that $Q^{1,0}_{g_t}(\Sigma,\partial\Sigma)\leq\sigma^{1,0}(\Sigma,\partial\Sigma)$ for each $t\in(-\varepsilon,\varepsilon)$.

As mentioned before, $\widetilde{g}_{0} =u_0^{\frac{4}{n-3}}g_0$ is a metric which has scalar curvature equal
to $\inf R^M$ and zero mean curvature on the boundary. Moreover, since $\Sigma$ is infinitesimally rigid and using identities (\ref{gauss}) and (\ref{free}),
we also have that $R_{g_0}=\inf R^M$ and $k_{g_0}=0$, then the Maximum Principle implies that $u_0\equiv 1$. We include the argument for completeness.
Define $w=u_0^{\frac{4}{n-3}}-1$. Therefore (\ref{minimi0}) is equivalent to

\begin{equation}
\left\{\begin{matrix}
\Delta_{\Sigma}w+h(x)w=0\ & \textrm{in $\Sigma$}\\
\frac{\partial w}{\partial \nu}=0 & \textrm{on $\partial\Sigma$,}
\end{matrix}\right.
\end{equation}
where $h(x)=\frac{\inf R^M(n-3)^2}{16(n-2)}(u_0+u_0^{\frac{n-2}{n-3}}+u_0^{\frac{n-1}{n-3}}+u_0^{\frac{n}{n-3}})u_0^{\frac{n-7}{n-3}}<0$. It follows from uniqueness for the Neumann problem
that $w\equiv0$ \cite[Theorem 3.6]{GT}. Hence, we must have $u_0\equiv 1$.

On the other hand, we also have that $\ell_0\equiv1$.
By continuity, we can find a positive constant  $K_1$ such that $\Psi(t)>K_1$ for all $t\in(-\varepsilon,\varepsilon)$.

Combining (\ref{eq5}) and equality (\ref{ineq}) we infer

\begin{eqnarray}\nonumber
H'(t)&\leq& -\frac{\inf R^M}{2K_1}(vol(\Sigma_t)^{\frac{2}{n-1}}-vol(\Sigma)^{\frac{2}{n-1}})\nonumber\\
&=&-\frac{\inf R^M}{K_1}\int_0^t\Big(\frac{d}{ds}vol(\Sigma_s)\Big)vol(\Sigma_s)^{\frac{3-n}{n-1}}ds.\nonumber
\end{eqnarray}

As a consequence of (\ref{RV}) we have
\begin{eqnarray}\label{respct}
H'(t)\leq -\frac{\inf R^M}{(n-1)K_1}\int_0^t vol(\Sigma_s)^{\frac{3-n}{n-1}}H(s)\int_{\Sigma}\ell_sd\sigma_sds.
\end{eqnarray}

Suppose by contradiction that there exists $t_0\in(0, \varepsilon)$ such that $H(t_0)>0$. Consider $\tau=\inf\{t\in[0, t_0];\;H(t)\geq H(t_0) \}$.
We claim that $\tau=0$. In fact, if $\tau>0$ the Mean Value Theorem implies that there exists $t_1\in(0,\tau)$ such that
\begin{equation}\label{meancur}
H'(t_1)=\frac{1}{\tau}H(\tau).
\end{equation}

Now we use (\ref{meancur}) in (\ref{respct}) to obtain

\begin{equation}\label{eqest}
H(\tau)\leq-\frac{\inf R^M \tau}{(n-1)K_1}\int_0^{t_1} H(s)\xi(s)ds\leq-\frac{\inf R^M \tau}{(n-1)K_1}\int_0^{t_1} H(\tau)\xi(s)ds,\nonumber\\
\end{equation}\nonumber
where $\xi(s)=vol(\Sigma_s)^{\frac{3-n}{n-1}}\big(\int_{\Sigma}\ell_sd\sigma_s\big)$ and we used, by definition of $\tau$,
that $H(t)\leqslant H(t_0)=H(\tau)$ for all $t\in[0,\tau]$.

We can also find a positive constant $K_0$ such that $\xi(t)<K_0^{-1}K_1$. Choosing
 $\varepsilon>0$ such that $\varepsilon^2<-\frac{(n-1)K_0}{\inf R^M}$ we get

\begin{equation}\nonumber
H(\tau)\leq -\frac{\inf R^M}{K_0(n-1)} H(\tau)\varepsilon^2<H(\tau),
\end{equation}
which gives a contradiction.

Since $\tau=0$, it follows that $H(0)\geq H(t_0)>0$, so we get again the desired contradiction. Therefore $H(t)\leq0$ for
$t\in[0,\varepsilon)$. In a similar way we deduce that $H(t)\geq0$ for $t\in(-\varepsilon,0]$.

By using (\ref{RV}), we conclude that $vol(\Sigma_t)\leq vol(\Sigma)$ for all $t\in(-\varepsilon,\varepsilon)$.

\begin{case} $\inf R^M=0$ and $\sigma^{1,0}(\Sigma,\partial\Sigma)\leqslant0.$
\end{case}

By definition of the Yamabe invariant, we have $H'(t)\leq0$ for every  $t\in(-\varepsilon,\varepsilon)$. Therefore $H(t)\leqslant H(0)=0$ for $t\in[0,\varepsilon)$
and $H(t)\geqslant H(0)=0$ for $t\in(-\varepsilon,0]$. Thus, $vol(\Sigma)\geq vol(\Sigma_t)\;\;\forall \,t\in(-\varepsilon,\varepsilon).$

\end{proof}

Finally, after these preparations, we are now able to complete the proof of the local splitting in Theorem \ref{princ}.

\begin{proposition}\label{final}
 If $\Sigma$ is infinitesimally rigid, then $\Sigma$ has a neighborhood in $M$ which is isometric to
 $((-\varepsilon,\varepsilon)\times \Sigma,dt^2+g)$ for some $\varepsilon>0$ and the induced metric $g$ on $\Sigma$ is Einstein.
\end{proposition}
\begin{proof}
Let $\Sigma_t\subset M$, $t\in(-\varepsilon,\varepsilon)$ be free boundary hypersurfaces given by Proposition \ref{folhea}.
From Proposition \ref{proimp} we conclude that $vol(\Sigma_t)\leq vol(\Sigma)$
for all $t\in(-\varepsilon,\varepsilon)$. But, since $\Sigma$ is locally volume-minimizing we obtain
$$
vol(\Sigma_t)=vol(\Sigma)
$$
for all $t\in(-\varepsilon,\varepsilon)$. In particular, each $\Sigma_t$ is volume-minimizing.
Therefore, each $\Sigma_t$ is infinitesimally rigid.

It follows from Lemma \ref{rambr} that, since the lapse function satisfies the homogeneous Neumann problem, $\ell_t$
is constant (as function of $t$) at each $\Sigma_t$. The function $\mu(t,x)=0$ and the vector field $N_t$ is parallel for all $(t,x)\in(-\varepsilon,\varepsilon)\times\Sigma$ (see \cite{MM} or \cite{N}) and its flow is the exponential map, i.e., $\mathfrak{f}(t, x)=\exp_x (tN(x))$ $\forall x \in\Sigma$ which
is an isometry for all $t\in(-\varepsilon,\varepsilon)$. Hence, the metric of $M$ near $\Sigma$ must split as $dt^2+g$.

\end{proof}

\section{Proof of Theorem \ref{quart}}\label{five}

We begin with the following infinitesimal rigidity which was inspired by \cite[Theorem 3.1]{Y}. (Compare Proposition \ref{propig} and \ref{propig2}).
\begin{proposition}\label{lastprop}
Let $M$ be an $n$-dimensional  Riemannian  manifold with scalar curvature $R^M\geq -n(n-1)$ and mean convex boundary. Assume that $M$
contains a two-sided compact properly embedded free boundary hypersurface $\Sigma$ such that $\sigma^{1,0}(\Sigma,\partial\Sigma)\leq0$.
If $\Sigma$ is $\mathcal{J}$-stable, then $R^M=-n(n-1)$ and $Ric(N,N)=-(n-1)$ along $\Sigma$, $\Sigma$ is umbilic,
$\sigma^{1,0}(\Sigma,\partial\Sigma)=0$, the mean curvature of $\partial\Sigma$ in $\Sigma$ is equal to zero,
$H^{\partial M}=0$ at every point of $\partial\Sigma$, $\Pi(N,N)=0$  along $\partial\Sigma$
and the induced metric on $\Sigma$ is Ricci flat.
\end{proposition}

\begin{proof}
From definition of $\mathcal{J}$-stability and identities (\ref{gauss}) and (\ref{free}) we infer
\begin{eqnarray}\nonumber
0&\leq&\int_{\Sigma}\Big(2\|\nabla \varphi\|_g^2+(R_g-(R^M+n(n-1))-\|\mathring{h}^{\Sigma}\|^2)\varphi^2\Big)d\sigma\\
& +&\int_{\partial\Sigma}2(-H^{\partial M}+\kappa_g)\varphi^2d\sigma_{\partial\Sigma},\nonumber
\end{eqnarray}
where $\mathring{h}^{\Sigma}=h^{\Sigma}-g$ is the trace free part of $h^{\Sigma}$.

By using $\frac{4(n-2)}{n-3}>2$ for all $n\geq4$, $R^M\geq-n(n-1)$, $\|\mathring{h}^{\Sigma}\|^2\geq0$ and that the boundary $\partial M$ is
mean convex we conclude

$$
\int_{\Sigma}\Big(a_n\|\nabla \varphi\|_g^2+R_g\varphi^2\Big)d\sigma+\int_{\partial\Sigma}2\kappa_g\varphi^2d\sigma_{\partial\Sigma}\geq0.
$$
Then $Q^{1,0}(\Sigma,\partial\Sigma)\geq0$  and, hence, by the definition of the Yamabe invariant (\ref{yamainv}) $\sigma^{1,0}(\Sigma,\partial\Sigma)\geq 0$. On the other hand, $\sigma^{1,0}(\Sigma,\partial\Sigma)\leq 0$
by supposition which implies $\sigma^{1,0}(\Sigma,\partial\Sigma)=0$. Moreover, we obtain $R^M=-n(n-1)$ and
$\|\mathring{h}^{\Sigma}\|^2=0$ along $\Sigma$. Essentially by the same argument used in Proposition \ref{propig}, the other assertions follow.
The equality (\ref{gauss}) and the fact that $\Sigma$ is Einstein imply that the induced metric on $\Sigma$ is Ricci flat, which completes the proof.
\end{proof}

Now, we prove a local warped product splitting result.

\begin{proof}[Proof of Theorem \ref{quart}]
 From Proposition \ref{lastprop}, we have that $\mathcal{L}=\Delta_{\Sigma}$. Then we can construct by Proposition \ref{folhea} a
foliation around $\Sigma$ by constant mean curvature free boundary hypersurfaces.

Since $\Sigma$ locally minimizes the functional $vol(\Sigma)-(n-1)\mathcal{V}(0),$ we must have $H(0)= n-1$. We need to show that
$H(t)\leq n-1$ for $t\in[0,\varepsilon)$. Otherwise, there exists $\tau\in(0,\varepsilon)$ such that $H(\tau)>n-1$
and (decreasing $\varepsilon$ if necessary) $H'(\tau)>0$.

Let $g_t$ be the induced metric on $\Sigma$ and $\widehat{g}_{\tau}=u_{\tau}^{\frac{4}{n-2}}g_{\tau}$ be
a conformally related metric with constant scalar curvature and boundary being a minimal hypersurface.
Taking in account that $R^M\geq -n(n-1)$ and $\|h^{\Sigma_{\tau}}\|^2\geq\frac{H(\tau)^2}{n-1}> n-1$,
we obtain that $R^M+\|h^{\Sigma_{\tau}}\|^2+H(\tau)^2>0$.
Then we deduce from (\ref{eq1}) that
$$
2H'(\tau)\ell_{\tau}^{-1}<-2\ell_{\tau}^{-1}\Delta_{\tau}\ell_{\tau}+R_{\tau}.
$$
Proceeding as in Proposition \ref{proimp}, we can show that

$$
0<H'(\tau)\Psi(\tau)<\sigma^{1,0}(\Sigma,\partial\Sigma),
$$
where $\Psi(\tau)= (\int_{\Sigma}u_{\tau}^{\frac{2(n-1)}{n-3}}d\sigma)^{\frac{3-n}{n-1}}\int_{\Sigma}\frac{u_\tau^2}{\ell_\tau}d\sigma_\tau$.
Since $\sigma^{1,0}(\Sigma,\partial\Sigma)\le0$ we arrive at a contradiction. Hence, $H(t)\leq n-1$ for $t\in [0,\varepsilon)$.

By the first variation formula of $vol(\Sigma)-(n-1)\mathcal{V}(0)$ (\ref{primvaria}), it follows that $\mathcal{J}'(t)\leq0$ for all
$t\in[0,\varepsilon)$. We must have $\mathcal{J}'(t)=0$ for $t\in[0,\varepsilon)$ since $\mathcal{J}$ achieves a minimum at $t=0$.
Hence $H(t)=n-1$ for all $t\in[0,\varepsilon)$. A similar argument shows that $H(t)= n-1$ for $t\in (-\varepsilon,0]$. By Lemma \ref{rambr} and the free
boundary condition, we obtain that each $\Sigma_t$ is $\mathcal{J}$-stable. Thus, we get that each $\Sigma_t$ is infinitesimally rigid in the sense of Proposition \ref{lastprop}.

Up to isometry, the metric in a sufficiently small neighborhood of $\Sigma$ can be written as $g_M=\ell_t dt^2+g_{t}$.
By Lemma \ref{rambr}, the lapse function is constant as function of $t$ on $\Sigma_t$. 
Then by a change of the coordinate $t$ we may assume that $\ell_t=1$.

The induced metric on $\Sigma_t$ evolves as
$$\frac{\partial}{\partial t}(g_{ij})_t=2\ell_t(g_{ij})_t.$$

Therefore
\[
g_{t}=e^{2t}g
\]
for all  $t\in(-\varepsilon,\varepsilon)$. Thus, we deduce that the induced metric  by $\mathfrak{f}(t,x)$ on $(-\varepsilon,\varepsilon)\times\Sigma$ is given as follows $dt^2+e^{2t}g$
that is Ricci flat in $M$.

\end{proof}


\begin{thebibliography}{25}

\bibitem{Al} S. Almaraz, An existence theorem of conformal scalar-flat metrics on manifolds with boundary, Pacific Journal of Mathematics.
248 (2010), 1-22.

\bibitem{A}L. C. Ambrozio, Rigidity of area-minimizing free boundary surfaces in mean convex three-manifolds, 
J. Geom. Anal. (2013) doi:10.1007/s12220-013-9453-2. Published electronically

\bibitem{Au}T. Aubin, Equations diff\'erentielles non lin\'eaires et probl\'eme de Yamabe concernant la courbure scalaire, J. Math.
Pures Appl. 55 (1976), 269-296.

\bibitem{BCE}J.L. Barbosa, M. Carmo, and J. Eschenburg, Stability of hypersurfaces of constant mean curvature in Riemannian manifolds, 
Math. Z. 197 (1988), 123-138.

\bibitem{B}S. Brendle, Rigidity phenomena involving scalar curvature, Surveys in Differential Geometry, volume XVII (2012), 179?202 


\bibitem{Br}H. Bray, The Penrose inequality in general relativity and volume comparison theorems involving scalar curvature,
PhD Thesis, Stanford University (1997).

\bibitem{BBEN}H. Bray, S. Brendle, M. Eichmair, and A. Neves, Area-minimizing projective planes in three-manifolds, Commun. Pure Appl. Math. 63 (2010),
1237-1247.

\bibitem{BBN} H. Bray, S. Brendle and A. Neves, Rigidity of area-minimizing two-spheres in three-manifolds, Comm. Anal. Geom. 18 (2010), 821-830.

\bibitem{C}M. Cai, Volume minimizing hypersurfaces in manifolds of nonnegative scalar curvature, Minimal surfaces, geometric analysis and symplectic geometry
(Baltimore,MD, 1999), Adv. Stud. Pure Math., vol. 34, Math. Soc. Japan, Tokyo, 2002, pp. 1-7.

\bibitem{CG}M. Cai and G. Galloway, Rigidity of area-minimizing tori in 3-manifolds of nonnegative scalar curvature, Comm. Anal. Geom. 8 (2000), 565-573.

\bibitem{CR} K. Castro and C. Rosales, Free boundary stable hypersurfaces in manifolds with density and rigidity
results, J. Geom. Phys. 79 (2014), 14-28.

\bibitem{CFP} J. Chen, A. Fraser and C. Pang, Minimal immersions of compact bordered Riemann surfaces with free boundary, arXiv:1209.1165.

\bibitem{E} J. Escobar, The Yamabe problem on manifolds with boundary, J. Differential Geom. 35 (1992), 21-84.

\bibitem{E1} J. Escobar, Conformal deformation of a Riemannian metric to a scalar flat metric with constant mean curvature at the boundary,
Ann. of Math. 136 (1992), 1-50.

\bibitem{E2}J. Escobar, Uniqueness theorems on conformal deformation of metrics, Sobolev inequalities and an eigenvalue estimate,
Commun. Pure Appl. Math. 43 (1990), 857-883.

\bibitem{GT}Gilbarg, D.  and Trudinger, N.S.: Elliptic partial differential equations of second order.
Second edition, Springer, 1983.

\bibitem{L}O. Ladyzhenskaia and N. Uralt'seva, Linear and quasilinear elliptic equations, Academic Press, New York (1968) 495 pp.

\bibitem{M}F. Marques, Existence results for the Yamabe problem on manifolds with boundary,
Indiana Univ. Math. J. 54 (2005), 1599-1620.

\bibitem{M1}F. Marques, Conformal deformation to scalar flat metrics with constant mean curvature on the boundary.
Commun. Anal. Geom. 15(2) (2007), 381-405.


\bibitem{MM}M. Micallef and V. Moraru, Splitting of 3-manifolds and rigidity of area-minimising surfaces, arXiv:1107.5346, to appear in Proc. Amer. Math. Soc.

\bibitem{M}V. Moraru, On Area Comparison and Rigidity Involving the Scalar Curvature, PhD. Thesis, University of Warwick. (2013).

\bibitem{Na}G. Nardi, Schauder estimate for solutions of Poissones equation with Neumann boundary condition, arXiv:1302.4103

\bibitem{N}I. Nunes, Rigidity of area-minimizing hyperbolic surfaces in three-manifolds, J. of Geom. Anal., published electronically 20 December 2011,
doi: 10.1007/s12220-011-9287-8.

\bibitem{RV}A. Ros and E. Vergasta, Stability for hypersurfaces of constant mean curvature with free boundary, Geom. Dedicata 56 (1995), 19-33.

\bibitem{S}R. Schwartz, Monotonicity of the Yamabe invariant under connect sum over the boundary,
Ann. Global Anal. Geom. 35,  (2009), 115-131.


\bibitem{Sc}R.M. Schoen, Conformal deformation of a Riemannian metric to constant scalar
curvature, J. Diff. Geo. 20 (1984), 479-495.

\bibitem{SY}R. Schoen, S.T. Yau, Existence of incompressible minimal surfaces and the topology of three-dimensional manifolds with nonnegative
scalar curvature, Ann. of Math. 110 (1979), 127-142.

\bibitem{SZ} Y. Shen and S. Zhu, Rigidity of stable minimal hypersurfaces, Math. Ann. 309,  (1997), 107-116.


\bibitem{Tr}N. Trudinger, Remarks concerning the conformal deformation of Riemannian structures on compact manifolds, Annali Scuola Norm.
Sup. Pisa, 22 (1968), 265-274.

\bibitem{Y} S.T. Yau, Geometry of three manifolds and existence of black hole due to boundary effect, Adv. Theor. Math. Phy., 5 (2001), 755-767.


\end{thebibliography}
\end{document}